\documentclass{article}
\usepackage[left=3cm,right=3cm,top=3cm,bottom=3cm]{geometry} 
\usepackage{amsmath,amssymb} 
\usepackage{amsthm}
\usepackage{graphicx}
\usepackage{float}
\usepackage{caption}
\usepackage{subcaption}

\newcommand*{\rom}[1]{\expandafter\@slowromancap\romannumeral #1@}
\newcommand{\sA}{\mathcal A}
\newcommand{\sC}{\mathcal C}
\newcommand{\sD}{\mathcal D}
\newcommand{\sE}{\mathcal E}
\newcommand{\sF}{\mathcal F}

\newcommand{\sL}{\mathcal L}

\newcommand{\sP}{\mathcal P}
\newcommand{\sS}{\mathcal S}
\newcommand{\sT}{\mathcal T}

\newcommand{\R}{\mathbb R}
\newcommand{\E}{\mathbb E}
\newcommand{\F}{\mathbb F}
\newcommand{\Prob}{\mathbb P}

\newtheorem{thm}{Theorem}
\newtheorem{mres}{Main Result}
\newtheorem{prop}{Proposition}

\newtheorem{lem}{Lemma}
\newtheorem{cor}{Corollary}

\newtheorem{ass}{Assumption}
\begin{document}

\title{Optimal Stopping and the Sufficiency of Randomized Threshold Strategies\thanks{
University of Warwick, Coventry, CV4 7AL. UK. Email: vicky.henderson@warwick.ac.uk, d.hobson@warwick.ac.uk, m.zeng@warwick.ac.uk. We would like to thank participants at the 10th Oxford-Princeton workshop (May 25-26, 2017) for helpful comments. Matthew Zeng is supported by a {Chancellor's International Scholarship} at the University of Warwick.
}}

\author{Vicky Henderson\qquad 
David Hobson\qquad
Matthew Zeng
 \\
}

\date{\today}
\maketitle

\begin{abstract}
In a classical optimal stopping problem the aim is to maximize the expected value of a functional of a diffusion evaluated at a stopping time.
This note considers optimal stopping problems beyond this paradigm. We study problems in which the value associated to a stopping rule depends on the law of the stopped process. If this value is quasi-convex on the space of attainable laws then it is well known result that it is sufficient to restrict attention to the class of threshold strategies. However, if the objective function is not quasi-convex, this may not be the case. We show that, nonetheless, it is sufficient to restrict attention to mixtures of threshold strategies.

\end{abstract}

\begin{center}
\line(1,0){440}
\end{center}

\section{Introduction and main results}

Let $Y=(Y_t)_{t \geq 0}$ be a time-homogeneous, continuous strong-Markov process. Let $\sT$ be the set of all stopping times, and let $\sT_T$ be the set of all (one- and two-sided) threshold stopping times, ie. stopping rules based on the first crossing of upper or lower thresholds. 
Let $V=V(\tau)$ be the value associated with a stopping rule $\tau$. 
Consider the optimal stopping problem associated with $V$, ie. the problem of finding
\begin{equation}
\label{eq:osp}
 V_*( \sS) = \sup_{\tau \in \sS} V(\tau)
\end{equation}
where $\sS$ is some set of stopping times (for example $\sS = \sT$ or $\sS= \sT_T$), and especially the problem of finding an optimizer for \eqref{eq:osp}.
We say the $V=V(\tau)$ is law invariant if, whenever $\sigma,\tau$ are stopping times, $\sL(Y_\sigma)= \sL(Y_\tau)$ implies that $V(\sigma)=V(\tau)$, where $\sL(Z)$ is the law of $Z$. It follows that $V(\tau)=H(\sL(Y_\tau))$ for some map $H$.

The following result is well-known, but we include it as a contrast to our result on the sufficiency of randomized threshold rules.
\begin{mres}[See Theorem~\ref{thm:main2} below]
Suppose $H$ is quasi-convex and lower semi-continuous. Then $V_*(\sT_T) = V_*(\sT)$.
\end{mres}

\begin{cor}
In the setting of Theorem~\ref{thm:main2}, in solving the optimal stopping problem \eqref{eq:osp} over the set of all stopping times it is sufficient to restrict attention to threshold rules.
\label{cor:A}
\end{cor}

As the canonical example, consider expected utility, whence $V(\tau) = \E[u( Y_\tau) ]$, for a continuous, increasing function $u$. Then $V$ is law invariant. Indeed $V(\tau)= H(\sL(Y_\tau))$ where $H(\zeta) = \int u(z) \zeta(dz)$. $H$ is quasi-convex and lower semi-continuous. In this example it is well known that
there is an optimal stopping rule which is of threshold form, see for example, Dayanik and Karatzas~\cite{DayanikKaratzas:03}. The fact that quasi-convexity means that there is no benefit from following randomized strategies
is well understood in the economics literature, see Machina~\cite{Machina:85} Camerer and Ho~\cite{CamererHo:94}, Wakker~\cite{Wakker:10} and He et al~\cite{HeHuOblojZhou:17}.

Recently there has been a surge of interest in problems which, whilst they have the law invariance property, do not satisfy the quasi-convex criterion. Two examples are optimal stopping under prospect theory (Xu and Zhou~\cite{XuZhou:13}), and optimal stopping under cautious stochastic choice (Henderson et al~\cite{HendersonHobsonZeng:17}).

Introduce the set $\sT_R$ of mixed or randomized threshold rules.

\begin{mres}[See Theorem~\ref{thm:main1} below] Suppose law invariance holds for $V$, but not quasi-convexity for $H$. Then
$V_*(\sT_T) \leq V_*(\sT_R) = V_*(\sT)$.
\end{mres}

We will show by example that the first inequality may be strict.

\begin{cor}
In the setting of Theorem~\ref{thm:main1}, in solving the optimal stopping problem \eqref{eq:osp} over the set of all stopping rules it is sufficient to restrict attention to randomized threshold rules, but it may not be sufficient to restrict attention to (pure) threshold rules.
\label{cor:B}
\end{cor}

It should be noted that we do not include discounting in our analysis since a problem involving discounting does not satisfy the law invariance property. Nonetheless, as is well known, the conclusion of Corollary~\ref{cor:A} remains true for the problem of maximizing discounted expected utility of the stopped process $V(\tau) = \E[ e^{- \beta \tau} u(Y_\tau)]$. However, in problems which go beyond the expected utility paradigm, there are often modelling issues which mitigate against the inclusion of discounting. For this reason, historically the literature has concentrated on problems with no discounting. Finding the optimal stopping rule is often already challenging in these models.

The significance of Corollary~\ref{cor:B} is as follows. In many classical models optimal stopping behavior involves stopping on first exit from an interval.
If decision makers are observed to stop at levels which have already been visited by the process, then this behavior is inconsistent with the classical optimal stopping model. However, our result implies that the converse is not true: if decision makers are observed to stop only when the process is reaching new maxima or minima, then it does not necessarily mean that they are maximizers of expected payoffs. Instead the decision criteria may be more complicated, and they may be utilizing a randomized threshold rule.

\section{Problem specification and the problem in natural scale}

We work on a filtered probability space $(\Omega, \sF, \F = \{ \sF_t \}_{t \geq 0} , \Prob)$. Let $Y= (Y_t)_{t \geq 0}$ be a $(\F, \Prob)$-stochastic process on this probability space with state space $I$ which is an interval. Let $\bar{I}$ be the closure of $I$. We suppose that $Y$ is a regular, time-homogeneous diffusion with initial value $Y_0=y$ such that $y$ lies in the interior of $I$.

Let $\sT$ be the class of all stopping times $\tau$ such that $\lim_{t \uparrow \infty} Y_{t \wedge \tau}$ exists (almost surely).
We introduce two subclasses of stopping times
\begin{itemize}
\item $\sT_T$, the subclass of (pure) threshold stopping times;
\item $\sT_R$, the subclass of randomised threshold stopping times.
\end{itemize}
Note that $\sT_T \subset \sT_R \subset \sT$.
The set of pure threshold stopping times includes stopping immediately and can be written as
\begin{equation}
 \sT_T = \sT \cap \left(\cup_{\beta \leq y \leq \gamma; \; \beta, \gamma \in \bar{I}^Y} \{ \tau_{\beta,\gamma} \} \right), \label{eq:TTdef}
\end{equation}
where $\tau_{a,b} = \inf_{u \geq 0} \{ u: Y_u \notin (a,b) \}$. Note that if $a = y$ or $b=y$ then $\tau_{a,b}=0$ almost surely, and that if $\sigma=\tau$ almost surely then we have $V(\sigma)=V(\tau)$. Hence we may suppose that $\tau \equiv 0$, the strategy of stopping immediately, lies in $\sT_T$.

In order to be able to define a sufficiently rich class of randomized stopping times we need to assume that $\F$ is larger than the filtration generated by $Y$.
\begin{ass}
$\sF_0$ is sufficiently rich as to include a continuous random variable, and the stochastic process $Y$ is independent of this random variable.
\label{ass:filtration}
\end{ass}

It follows from the assumption that for any probability measure $\zeta$ on $\sD = ([-\infty,y] \cap \bar{I}) \times ([y,\infty]\cap \bar{I})$ there exists an  $\sF_0$-measurable random variable $\Theta = \Theta_\zeta = (A_\zeta, B_\zeta)$ such that $(A_\zeta, B_\zeta)$ has law $\zeta$. For a set $\Gamma$ let $\sP(\Gamma)$ be the set of probability measures on $\Gamma$. Then for any $\zeta \in \sP(\sD)$ we can define the randomised stopping time $\tau_\zeta$ as the first time $Y$ leaves a random interval, where the interval is chosen at time 0 with law $\zeta$. Then $\tau_\zeta = \tau_{A_\zeta, B_\zeta} = \inf \{ u : Y_u \notin (A_\zeta,B_\zeta)\}$. The set of randomized threshold rules $\sT_R$ is given by
\begin{equation}
\sT_R = \sT \cap \left( \{ \tau_\zeta : \zeta \in \sP(\sD) \} \right).
\label{eq:TRdef}
\end{equation}

Our analysis is focussed on problems in which the value associated with a stopping rule depends only on the law of the stopped process. Let $Q(\sS)= \{ \mu : \mu = \sL(Y_\tau) , \tau \in \sS \}$.

\begin{ass}[Law invariance]
\label{ass:lip}
$V$ is law invariant, ie
$V(\tau) = H(\sL(Y_\tau))$ for some function $H : Q(\sT) \mapsto \R$.
\end{ass}

Given that the value associated with a stopping rule is law invariant, one natural approach to finding the optimal stopping time is to try to characterize $Q(\sS)$.
Often, the best way to do this is via a change of scale. Let $s$ be a strictly increasing function such that $X = s(Y)$ is a local martingale. (Such a function $s$ exists under very mild conditions on $Y$ see, for example Rogers and Williams~\cite{RogersWilliams:00}, and is called a scale function. For example, if $Y$ solves the SDE $dY_t = \sigma(Y_t) dB_t + \xi(Y_t) dt$ for smooth functions $\sigma$ and $\xi$ with $\sigma > 0$ then $s=s(z)$ is a solution to $\frac{1}{2} \sigma(z)^2 s'' + \xi(z) s' = 0$. Note that if $s$ is a scale function then so is any affine transformation of $s$ and so we may choose any convenient normalization for $s$.) Let $I^X = s(I)$ and let $\bar{I}^X$ be the closure of $I^X$. Then $X$ is a regular, time-homogenous local-martingale diffusion on $I^X$ with initial value $x=s(y)$.

Set $Q^X(\sS) = \{ \nu : \nu = \sL(X_\tau) , \tau \in \sS \}$. Then if $\sL(X_\tau) = \nu$ we have $\sL(Y_\tau) = \nu \sharp s$ where $(\nu \sharp s)(D) = \nu(s(D))$. It follows that $\nu \in Q^X(\sS)$ if and only if $\nu \sharp s \in Q(\sS)$ and hence
\begin{equation}
\label{eq:QQX}
Q(\sS) = \{ \nu \sharp s ; \nu \in Q^X(\sS) \}.
\end{equation}
Thus, if we can characterize $Q^X(\sS)$ then we can also characterize $Q(\sS)$. Moreover, defining $H^X : Q^X(\sT) \mapsto \R$ by $H^X(\nu) = H(\nu \sharp s)$ we have
$V_*(\sS) = \sup_{\mu \in Q(\sS)} H(\mu) = \sup_{\nu \in Q^X(\sS)} H^X(\nu)$. The problem of optimizing over stopping laws for the problem with $Y$ becomes
a problem of optimizing over the possible laws of the stopped process $X$ in natural scale.

Note that $\tau_{a,b} = \inf_{u \geq 0} \{ u : Y_u \notin (a,b) \} = \inf_{u \geq 0} \{ u : X_u \notin (s(a),s(b)) \} =: \tau^X_{s(a),s(b)}$. Hence $\sT_T$ has the alternative representation
\[ \sT_T = \sT \cap \left(  \cup_{\beta \leq x \leq \gamma; \; \beta, \gamma \in \bar{I}^X} \{ \tau^X_{\beta,\gamma} \} \right) , \]
and the set of threshold stopping times for $Y$ is the set of threshold stopping times for $X$. Similarly,
$\sT_R$ can be rewritten as $\sT_R = \sT \cap (\{ \tau^X_\eta  : \eta \in \sP(\sD^X) \})$ where $\sD^X = ([-\infty,x] \cap \bar{I}^X) \times ([x, \infty) \cap \bar{I}^X))$ and
\[ \tau^X_\eta = \inf_{u \geq 0} \{u : X_u \notin (A_\eta, B_\eta) \mbox{where $(A_\eta, B_\eta)$ has law $\eta$} \} . \]

\section{Characterizing the possible laws of the stopped process in natural scale}
If $X=s(Y)$ is in natural scale then the state space of $X$ is an interval $I^X = s(I)$ and $X_0 = x := s(y)$. There are four cases:
\begin{enumerate}
\item $I^X$ is bounded;
\item $I^X$ is unbounded above but bounded below;
\item $I^X$ is bounded above but unbounded below;
\item $I^X$ is unbounded above and below.
\end{enumerate}
The third case can be reduced to the second by reflection. The first case is generally similar to the second case, and typically the proofs are similar but simpler. The final case is degenerate and will be treated separately. In the main text we will mainly present arguments for the second case (with the other cases covered in an appendix), but results will be stated in a form which applies in all cases.

Henceforth, in the main text we suppose $I^X$ is bounded below, but unbounded above. Without loss of generality we may assume $I^X=(0,\infty)$ or $[0,\infty)$. Then $X$ is a non-negative local martingale and hence a super-martingale. Moreover, $\lim_{t \rightarrow \infty} X_t$ exists. Hence $\sT$ includes stopping rules which take infinite values and on $\{\tau = \infty \}$ we set $X_\tau = \lim_{t \rightarrow \infty} X_t=0$. In this case $\sT$ is the set of all stopping times and the intersection with $\sT$ in the definitions \eqref{eq:TTdef} and \eqref{eq:TRdef} is not necessary.
By Fatou's lemma and the super-martingale property
\[ \E[X_\tau] = \E[ \lim_{t \rightarrow \infty} X_{t \wedge \tau}] \leq \liminf_{t \rightarrow \infty} \E[X_{t \wedge \tau}] \leq x .\]
In particular, if we set $\sP_{\leq x} =  \{ \nu \in \sP([0,\infty)) : \int z \nu(dz) \leq x \}$ then $Q^X(\sT) \subseteq \sP_{\leq x}$.


\begin{lem}
$Q^X(\sT) = Q^X(\sT_R)$.
\label{lem:Q=}
\end{lem}

\begin{proof}
Here we prove the lemma in the case where $I^X$ is bounded below. We show that $Q^X(\sT) = Q^X(\sT_R)=\sP_{\leq x}$.
Given $\nu \in \sP_{\leq x}$ the aim is to find a stopping time $\tau \in \sT_R$ such that $\sL(X_\tau)= \nu$. The task of finding general stopping times with $\sL(X_\tau) = \xi$ for given $\xi \in \sP(\overline{I}^X)$ is known as the Skorokhod embedding problem (Skorokhod~\cite{Skorokhod:65}). In fact we use an extension of an embedding due to Hall~\cite{Hall:85}, see also Durrett~\cite{Durrett:91}. The extension relates to the fact that we allow for target laws which have a different mean to the initial value of $X$, whereas the Hall embedding assumes $\int z \nu(dz) = x$. The Hall embedding, and the extension we give, are mixtures of threshold strategies.

Suppose $\nu$ is an element of $\sP_{\leq x}$ (and $\nu$ is not a point mass at $x$). The case of $\nu = \delta_x$ corresponds to the (threshold) stopping time $\tau=0$. Let $G$ be the (right-continuous) quantile function of $\nu$. We have
$x \geq \int z \nu(dz) = \int_{(0,1)}G(u) du$. In particular, unless $\lim_{u \uparrow 1}G(u) \leq x$ there exists a unique solution $v^* \in [0,1)$ to $\int_v^1 [G(w) - x] dw = 0$. Let $z^* = G(v^*)\leq x$. If $\lim_{u \uparrow 1}G(u) \leq x$ then set $v^*=1$ and $z^* = \lim_{u \uparrow 1}G(u)$.

Let $\nu_0$ be the measure of size $v^*$ such that $\nu_0([0,z)) = v^* \wedge \nu([0,z))$. Then $\nu_0$ has support contained in $[0,z^*]$.
Let $\nu_1$ be the measure of size $1-v^*$ such that $\nu_1([0,z)) =  (\nu([0,z))-v^*)^+$. Then $\nu_1$ has support in $[z^*,\infty)$ and barycentre $x$. Moreover $\nu = \nu_0 + \nu_1$.

Define $c = \int_x^\infty (y-x) \nu(dy)$. By construction, $c = \int_x^\infty (y-x) \nu_1(dy)$ and we have
from the fact that $\nu_1$ has barycentre $x$ that $\int_{z^*}^\infty (y-x) \nu_1(dy)=0$ and hence
\begin{equation}
c = \int_{z^*}^x (x-y) \nu_1(dy).
\label{eq:c}
\end{equation}

Let $\eta \in \sP([0,x] \times(x,\infty])$ be given by
\[ \eta(da,db)  = \nu_0(da) I_{ \{ 0 \leq a \leq z^* \} } I_{ \{ b = \infty \} }
                  + \nu_1(da) \nu_1(db) \frac{(b-a)}{c} I_{ \{ z^* \leq a \leq x < b < \infty  \} } . \]
Note first that $\eta$ is a probability measure:
\begin{eqnarray*}
\lefteqn{ \int_{0 \leq a \leq x} \int_{x < b \leq \infty} \eta(da,db) } \\
 & = & v^* 
      + \int_{z^* \leq a \leq x} \nu_1(da) \int_{x < b < \infty} \frac{b-x}{c} \nu_1(db)
      + \int_{z^* \leq a \leq x} \frac{x-a}{c} \nu_1(da) \int_{x < b < \infty}  \nu_1(db) \\
     &  = & v^* +  \int_{z^* \leq a \leq x} \nu_1(da) +  \int_{x < b < \infty}  \nu_1(db) = v^* + \nu_1([z^*,\infty)) = 1
\end{eqnarray*}
where we use the definition of $c$ and \eqref{eq:c} in going from the second line to the third.

It remains to show that $\sL(X_{\tau^X_\eta}) = \nu$.
Let $f$ be a bounded test function. Then, using the fact that if $b=\infty$ then $X_{\tau^X_{a,\infty}}=a$, and the definition of $c$ and \eqref{eq:c} for the penultimate line,
\begin{eqnarray*} \E[f( X_{\tau^X_\eta})] & = & \int \int \eta(da,db) \E[f(X_{\tau^X_{a,b}})] \\
& = & \int \nu_0 (da) f(a) + \int_{z^* \leq a \leq x}  \int_{x < b < \infty} \nu_1(da) \nu_1(db) \frac{b-a}{c} \left[ f(a)\frac{(b-x)}{b-a} + f(b) \frac{(x-a)}{b-a} \right] \\
& = & \int \nu_0 (da) f(a) + \int_{z^* \leq a \leq x} \nu_1(da) f(a) \int_{x < b<\infty} \nu_1(db) \frac{(b-x)}{c} \\
    && \hspace{30mm} + \int_{z^* \leq a \leq x} \frac{(x-a)}{c} \nu_1(da) \int_{x < b<\infty} f(b) \nu_1(db) \\
& = & \int_{0 \leq z \leq z^*} f(z) \nu_0(dz) + \int_{z^* \leq z \leq x} f(z) \nu_1(dz)+ \int_{x < z} f(z) \nu_1(dz) \\
& = & \int f(z) \nu(dz).
\end{eqnarray*}
Hence $\sL(X_{\tau_\eta}) = \nu$ as required.
\end{proof}

Let $\chi_{a,b} = \frac{b-x}{b-a} \delta_a + \frac{x-a}{b-a} \delta_b$.
Then $\chi_{a,b}$ is the law of $X_{\tau^X_{a,b}}$. Moreover, $\sL(X_{\tau^X_{a, \infty}}) = \delta_a$. Then,
\[ Q^X(\sT_T) = \left( \cup_{0 \leq a \leq x} \delta_x \right) \cup \left( \cup_{0 \leq a < x <b <\infty} \chi_{a,b} \right). \]

\section{Sufficiency of mixed threshold rules}

Our main result is that in a large class of problems it is sufficient to search over the class of mixed threshold rules.
\begin{thm}
\label{thm:main1}
Suppose $Y$ is a regular, time-homogeneous diffusion. Suppose the law invariance property holds (Assumption~\ref{ass:lip}) and that the filtration is sufficiently rich (Assumption~\ref{ass:filtration}). Then $V_*(\sT) = V_*(\sT_R)$.
\end{thm}

\begin{proof}
Since $Q^X(\sT) = Q^X(\sT_R)$ (Lemma~\ref{lem:Q=}) we have $Q(\sT) = Q(\sT_R)$. Then
\[ V_*(\sT) = \sup_{\mu \in Q(\sT)} H(\mu) = \sup_{\mu \in Q(\sT_R)} H(\mu) = V_*(\sT_R). \]
\end{proof}

Note that it is not our claim that every optimal stopping rule is a mixed threshold rule. Typically, at least in the case where $V(\sT_T) < V(\sT)$, there will be other optimal stopping rules which are not of threshold type.

\subsection{Examples}
\subsubsection{Rank dependent utility and optimal stopping}
Let $Z$ be a non-negative random variable. Let $v:[0,\infty) \mapsto [0,\infty)$ be an increasing, differentiable function with $v(0)=0$. Then the expected value of $v(Z)$ can be expressed as $\E[v(Z)] = \int_0^\infty v'(z) \bar{F}_Z(z) dz$. Under rank-dependent utility (Quiggin~\cite{Quiggin:82}) or probability weighting (Tversky and Kahneman~\cite{TverskyKahneman:96}) the prospect value $\sE_v(Z)$ of $Z$ is
\[ \sE_v(Z) = \int_0^\infty v'(z) w(\bar{F}_Z(z)) dz \]
where $w :[0,1] \mapsto [0,1]$ is an increasing, differentiable probability weighting function. Writing $G_Z=F_Z^{-1}$ for the quantile function of $Z$, then after a change of variable and integration by parts we have (see Xu and Zhou~\cite[Lemma 3.1]{XuZhou:13}) the alternative representation
\[ \sE_v(Z) = \int_0^1 w'(1-u) {G}_Z(u) du. \]

Now let $Y=(Y_t)_{t \geq 0}$ be a non-negative diffusion and consider the problem of maximizing over stopping times the prospect value of the stopped process $Y$, ie of finding
\begin{equation}
\label{eq:rdu}
 \sup_{\tau \in \sT} \sE_v(Y_\tau).
\end{equation}
Clearly the prospect value depends on the stopping time only through the law of the stopped process. Hence it is sufficient to characterize the optimal target distribution, for example via its quantile function. Xu and Zhou~\cite{XuZhou:13} solve for the optimal quantile function in several cases. One relevant case is the following:
\begin{prop}[Xu and Zhou~\cite{XuZhou:13}]
Suppose $Y$ is in natural scale and has state space $[0,\infty)$and initial value $y$. Suppose $v$ and $w$ are concave. Suppose there exists $\lambda^* \in (0,\infty)$ which solves
\[ \int_0^1 (v')^{-1} \left( \frac{\lambda^*}{w'(1-u)} \right) du = y .\]
Then the quantile function of the optimal stopping distribution is
$G^*(u) = (v')^{-1} \left( \frac{\lambda^*}{w'(1-u)} \right)$.
\label{prop:pt}
\end{prop}

Xu and Zhou~\cite{XuZhou:13} point out that although there is a unique optimal prospect there are infinitely many stopping rules which attain this prospect.
They advocate the use of the stopping rule based on the Az\'ema-Yor stopping time~\cite{AzemaYor:79}, in which case the stopping rule has a drawdown feature, and involves stopping the first time the process falls below some function of the maximum. Our main result says that there is also a randomized threshold rule which is optimal.

\subsubsection{Cautious stochastic choice}
Given a process $Y$ and a utility function $u$ the certainty equivalent associated with a stopping time $\tau$ is $\sC_u(\tau) = u^{-1}(\E[u(Y_\tau)])$. The idea in Cautious stochastic choice (Cerreia-Vioglio et al~\cite{CerreiaVioglio:15}) is that agents use multiple utility functions and evaluate an outcome in a robust manner as the least favorable of the individual certainty equivalents. If the set of utility functions is $\{ u_\alpha \}_{\alpha \in \sA}$, and if we write $\sC_\alpha$ as shorthand for $\sC_{u_\alpha}$ then the CSC value of a stopping rule is
\begin{equation}
CSC(\tau) = \inf_{\alpha \in \sA} \sC_\alpha(\tau) = \inf_{\alpha \in \sA} u_\alpha^{-1} (\E[u_\alpha(Y_\tau)]) ,
\label{eq:csc}
\end{equation}
and an optimal stopping rule is the one which maximizes the CSC value.

Clearly the CSC value of a stopping rule depends only on the law of $Y_\tau$. Moreover, suppose $\sA = \{\alpha, \beta \}$ and suppose $u_{\alpha}$ and $u_{\beta}$ are strictly increasing and continuous with strictly increasing and continuous inverses. Suppose further that there exist $\tau_1$ and $\tau_2$ and $\tilde{y}$ such that $u_\alpha^{-1}(\E[u_\alpha(Y_{\tau_1})]) > \tilde{y} > u_\beta^{-1}(\E[u_\beta(Y_{\tau_1})])$ and $u_\alpha^{-1}(\E[u_\alpha(Y_{\tau_2})]) < \tilde{y} < u_\beta^{-1}(\E[u_\beta(Y_{\tau_2})])$. 
Let $\tau^\theta$ be a mixture of $\tau_1$ and $\tau_2$, defined such that if $Z$ is a $\sF_0$-measurable random variable taking values in $\{1,2\}$ with $\Prob(Z=1)=\theta$ then $\tau^\theta = \tau_Z$.
Then for ${\gamma \in \sA}$, $\sC_{\gamma}(\tau^\theta) = u_\gamma^{-1} (\theta \E[u_\gamma(Y_{\tau_1})] + (1-\theta) \E[u_\gamma(Y_{\tau_2})])$ is a continuous function of $\theta$. Moreover, $\sC_{\alpha}(\tau^\theta)$ is strictly increasing in $\theta$ and $\sC_{\beta}(\tau^\theta)$ is strictly decreasing. By our assumptions it follows that the best choice $\theta^*$ of $\theta$ is such that $\sC_{\alpha}(\tau^{\theta^*}) = \sC_{\beta}(\tau^{\theta^*})$ and then $\theta^* \in (0,1)$ and
$CSC(\tau^{\theta^*}) > \max \{ CSC(\tau_1), CSC(\tau_2)\}$.

In particular, the value associated with a stopping rule is not quasi-convex. By the analysis of this section, in searching for an optimal stopping rule it is sufficient to restrict attention to randomized threshold rules, but we cannot expect in general that there is a pure threshold rule which is optimal.
For a deeper study of optimal stopping in the context of Cautious stochastic choice see Henderson et al~\cite{HendersonHobsonZeng:17}.

\section{Sufficient conditions for the optimality of pure threshold rules}
In this section we argue that if the value associated with a stopping rule is law invariant, and if $H$ is quasi-convex and lower semi-continuous then pure threshold rules are optimal.

Recall that $H$ is quasi-convex if $H(\lambda \mu_1 + (1-\lambda) \mu_2) \leq \max \{ H(\mu_1), H(\mu_2) \}$ for $\lambda \in(0,1)$. It follows by induction that if $\mu = \sum_{i=1}^N \lambda_i \mu_i$ where $\lambda_i \geq 0$,  $\sum_{i=1}^N \lambda_i=1$ and $\mu_i \in Q(\sT)$ then
\begin{equation}
\label{eq:qc}
 H(\mu) \leq \max_{1 \leq i \leq N} H(\mu_i) \leq \sup_{\tilde{\mu} \in Q(\sT)} H(\tilde{\mu}) .
\end{equation}
Recall also that if $H$ is lower semi-continuous and $\mu_n \Rightarrow \mu$ then $H(\mu) \leq \lim \inf H(\mu_n)$. In fact we do not require $H(\mu) \leq \lim \inf H(\mu_n)$, but rather the weaker condition $H(\mu) \leq \lim \sup H(\mu_n)$.

\begin{lem}
Suppose $\nu \in Q^X(\sT)$ consists of finitely many atoms. Then there exists $\eta \in \sP(\sD^X)$ such that $\eta$ consists of finitely many atoms and $\sL(X_{\tau^X_\eta})=\nu$.
\label{lem:atoms}
\end{lem}

\begin{proof}
It follows from the construction in the proof of Lemma~\ref{lem:Q=} that if $\mu$ is purely atomic then so is $\eta$.
\end{proof}

\begin{lem}
Let $\nu$ be an element of $Q^X(\sT)$. Then there exist $(\eta_n)_{n \geq 1}$ such that $\eta_n$ has finite support for each $n$ and such that $\sL(X_{\tau^X_{\eta_n}}) \Rightarrow \nu$.
\label{lem:approx}
\end{lem}

\begin{proof}
Since $\nu \in Q^X(\sT)=Q^X(\sT_R)$ there exists $\eta$ such that $\sL(X_{\tau^X_\eta}) = \nu$. Let $(\eta_n)_{n \geq 1}$ be a sequence of measures with finite support such that $\eta_n \Rightarrow \eta$. Then for $f:[0,\infty) \mapsto \R$ a bounded continuous test function define
$\tilde{f}:[0,x] \times [x,\infty)$ by $\tilde{f}(a,b) = f(a) \frac{b-x}{b-a} + f(b)\frac{x-a}{b-a}$ for $a<b$ with $\tilde{f}(x,x)=f(x)$.
Then, since $\tilde{f}$ is bounded and continuous
\[
\E[f(X_{\tau^X_{\eta_n}})]  =  \int \int \eta_n(da, db) \tilde{f}(a,b)
 \rightarrow  \int \int \eta(da, db) \tilde{f}(a,b) = \E[f(X_{\tau^X_{\eta}})]
\]
and it follows that $\nu_n := \sL(X_{\tau^X_{\eta_n}}) \Rightarrow \nu$.
\end{proof}

\begin{thm}
\label{thm:main2}
Suppose $Y$ is a regular, time-homogeneous diffusion. Suppose the law invariance property holds (Assumption~\ref{ass:lip}).
Suppose that $H$ is quasi-convex and lower semi-continuous.
Then $V_*(\sT) = V_*(\sT_T)$.
\end{thm}

\begin{proof}
Clearly $V_*(\sT) \geq V_*(\sT_T)$.

For any $\mu_n$ with finite support we can define $\nu_n = \mu_n \sharp s^{-1}$. Then we can find a measure $\eta_n$ with finite support such that
$\sL(X_{\tau^X_{\eta_n}}) = \nu_n$. Moreover $\nu_n$ can be decomposed as a convex combination
\[ \nu_n = \sum_{i=1}^N \gamma_i \chi_{a_i,b_i} + \sum_{j=1}^M \lambda_j \delta_{a_j}. \]
Then, since $H$ is quasi-convex,
\begin{eqnarray*} H(\mu_n) & \leq & \left( \max_{1 \leq i \leq N} H(\chi_{a_i,b_i} \sharp s^{-1}) \right) \vee
 \left( \max_{1 \leq j \leq M} H(\delta_{s^{-1}(a_j)}) \right) \\
 & \leq & \left( \sup_{0 \leq a \leq x \leq b<\infty} H(\chi_{a,b} \sharp s^{-1}) \right) \vee
 \left( \sup_{0 \leq a \leq x} H(\delta_{s^{-1}(a)}) \right) = V_*(\sT_T).
\end{eqnarray*}
Then, for $\tau \in \sT$, if $\mu = \sL(Y_\tau)$ and if $\mu_n \Rightarrow \mu$
\[ V_\tau = H(\mu) \leq \limsup H(\mu_n) \leq V_*(\sT_T). \]
Hence $V_*(\sT) \leq V_*(\sT_T)$.
\end{proof}

\section{Discussion}
In classical optimal stopping problems involving maximizing expected utility the optimal strategy is a threshold rule and involves stopping the first time that the process leaves an interval. However, in more general settings the optimal strategy may be more sophisticated. In some settings, for example those involving regret (Loomes and Sugden~\cite{LoomesSugden:82}) the optimal stopping rule may depend on some functional of the path (for example the maximum price to date). But, as argued here, for a large class of problems the payoff depends only on the distribution of the stopped process, and then there are many optimal stopping rules, some of which take the form of randomized threshold rules.  In this article we have utilized (an extended version of) the Hall solution of the Skorokhod embedding problem (Hall~\cite{Hall:85}) to give our randomized threshold rule, but there are other solutions of the Skorokhod embedding which can also be viewed as mixed threshold rules, including the original solution of Skorokhod~\cite{Skorokhod:65} and the solution of Hirsch et al~\cite{HirschProfettaRoynetteYor:11}.

The idea that if the objective is expressed in terms of a function which is not quasi-convex then agents may want to use randomised strategies is well appreciated in static settings. In a dynamic setting He et al~\cite{HeHuOblojZhou:17} argue that in binomial-tree, probability-weighted model of a casino (Barberis~\cite{Barberis:14}) gamblers may prefer path-dependent strategies over strategies which are defined via a partition of the set of nodes into those at which the gambler stops and those at which he continues. (See also Ebert and Strack~\cite{EbertStrack:16} and Henderson et al~\cite{HendersonHobsonTse:17} for  discussion of a related optimal stopping problem with probability weighting based on a diffusion process.) He et al~\cite{HeHuOblojZhou:17} argue further that the path-dependent strategy can be replaced by a randomized strategy under which the decision about whether to stop at a node depends not on the path history but rather the realization of an independent uniform random variable. This preference for randomization mirrors our result, but takes a different form. In our perpetual problem the agent chooses a randomized pair of levels and then follows a threshold strategy based on these levels. In He et al~\cite{HeHuOblojZhou:17} a zero-one decision about whether to stop at a node is replaced by a probability of continuing, and the stopping rules which arise are not randomized threshold rules.

Many optimal stopping models in the economics literature predict that the agent will stop on first exit from an interval, which necessarily involves stopping either at the current maximum or the current minimum. If instead, observed behavior includes stopping at levels which are not equal to one of the running extrema of the process then this is evidence against the model. (Strack and Viefers~\cite{StrackViefers:17} present experimental evidence from a laboratory game that players do not follow threshold strategies - instead players visit the same price three times on average before stopping.) But, our results imply that the converse is not true. Even if agents only ever take a decision to sell at a time when the process is at a new maximum or new minimum, this does not necessarily mean that agents are following a pure threshold rule. They could have any target distribution, as for example in Proposition~\ref{prop:pt}, but be realizing this target distribution via a randomized threshold rule.

\appendix
\section{Extension to other state spaces for the process in natural scale}
\subsection{The range of $X$ is unbounded below but bounded above}

In this case we may assume that ${I}^X = (-\infty,0)$ or $(-\infty, 0]$.
The analysis goes through almost unchanged except that now $X$ is a convergent sub-martingale and $Q^X(\sT) = Q(\sT_R) = \sP_{\geq x}$
where $\sP_{\geq x} = \{ \nu \in \sP((-\infty,0]) : \int z \nu(dz) \geq x \}$.

\subsection{The range of $X$ is bounded}
Suppose $X$ is bounded. In this case $Q(\sT) = Q(\sT_R) = \sP_{=x}$ where $\sP_{=x} = \{ \nu \in \sP(\bar{I}^X) : \int z \nu(dz) = x \}$. To see this note that $X$ is a uniformly integrable martingale and not just a super-martingale. Therefore we must have $\E[X_\tau] = \lim \E[X_{\tau \wedge t}] = x$ and hence $Q(\sT) \subseteq \sP_{=x}$.
Conversely, by the same argument as in Lemma~\ref{lem:Q=}, but this time with $v^*=0$ and $\nu_1 \equiv \nu$, we deduce that for any $\nu \in \sP_{=x}$ there exists a randomization $\eta$ such that $\sL(X_{\tau^X_\eta}) = \nu$. It follows that $Q(\sT) = Q(\sT_R) = \sP_{=x}$.

The proofs of Lemma~\ref{lem:atoms}, Lemma~\ref{lem:approx} and Theorem~\ref{thm:main1} go through unchanged.

\subsection{The range of $X$ is $\R$}
Now suppose $I^X$ is unbounded above and below. By the Rogozin trichotomy (Rogozin~\cite{Rogozin:66})
$-\infty = \lim \inf_t X_t < x < \lim \sup_t X_t = \infty$ and $\lim_{t \uparrow \infty} X_t$ does not exist. In this case we must restrict $\sT$ to the set of stopping times with $\Prob(\tau < \infty) = 1$. In the main text we set $\sT_T = \sT \cap \left( \cup_{\beta \leq y \leq \gamma, \beta, \gamma \in \bar{I}^Y} \{ \tau_{\beta,\gamma} \} \right)$ but we could equivalently write $\sT_T =  \cup_{(\beta,\gamma) \in \sD_0} \{ \tau_{\beta,\gamma} \}$, where
$\sD_0 = ([-\infty, y] \cap \bar{I}^Y) \times  ([y,\infty] \cap \bar{I}^Y) \setminus \{ s^{-1}(-\infty), s^{-1}(\infty) \}$. We have to exclude the threshold rule $\tau_{s^{-1}(-\infty), s^{-1}(\infty)}$ since $\tau_{s^{-1}(-\infty), s^{-1}(\infty)} = \infty$ almost surely and $Y_\infty$ is not defined. In terms of threshold rules $\tau^X_{a,b}$ for $X$ we allow $a = -\infty$ or $b = \infty$ but not both. Then $\sT_T = \{ \tau_{\beta,\gamma} : (\beta,\gamma) \in \sD^X_0) \}$ where $\sD^X_0 = \sD^X \setminus \{-\infty,\infty\} = [\infty,x] \times [x,\infty] \setminus \{-\infty,\infty \}$.

In the definition of randomized threshold rules we can write $\sT_R = \{ \tau_\zeta : \zeta \in \sP(\sD_0) \}$ where $\sD_0$ is as above and similarly
$\sT_R = \{ \tau^X_\eta : \eta \in \sP(\sD^X_0) \}$.

When $I^X=\R$ we claim that we have $Q^X(\sT) = Q^X(\sT_R) = \sP(\R)$. Since stopping times are finite almost surely we must have $Q^X(\sT) \subseteq \sP(\R)$ so it is sufficient to show that for any $\nu \in \sP(\R)$ we have $\nu \in Q^X(\sT_R)$. Given $\nu \in \sP(\R)$ let $A_\nu$ be a $\sF_0$-measurable random variable with law $\nu$ and set $\tau = \inf  \{u: X_u = A_\nu \}$. Then $\sL(X_\tau) = \sL(A_\nu) = \nu$.

The proofs of Lemma~\ref{lem:atoms}, Lemma~\ref{lem:approx} and Theorem~\ref{thm:main1} go through unchanged.

\subsection{Other results}
\begin{proof}[Proof of Proposition~\ref{prop:pt}]
A proof is given in Xu and Zhou~\cite[Theorem 5.1]{XuZhou:13}, but since it is short, elegant and pertinent to our main results we include it here. From the characterization of $Q(\sT)$ we have that a quantile function must satisfy $\int_0^1 G(u) du \leq y$. By construction $G^*$ has this property, and since $v'$ and $w'$ are decreasing, $G^*$ is increasing. Hence $G^*$ has the properties required of a quantile function of a distribution which
can be obtained by stopping $Y$. On the other hand, for any non-negative function $G$ with $\int_0^1 G(u) du \leq y$,
\begin{eqnarray*}
\int_0^1 w'(1-u) v(G(u)) du & = & \int_0^1 [w'(1-u) v(G(u)) - \lambda^* G(u)] du + \lambda^* \int_0^1 G(u) du \\
  &\leq &  \int_0^1 \sup_{g>0} [w'(1-u) v(g) - \lambda^* g] du + \lambda^* y \\
  & = &  \int_0^1  [w'(1-u) v(G^*(u)) - \lambda^* G^*(u)] du + \lambda^* y = \int_0^1 w'(1-u) v(G^*(u)) du .
\end{eqnarray*}
\end{proof}


\begin{thebibliography}{99}

\bibitem{AzemaYor:79} Az\'{e}ma J. and M. Yor, 1979, {Une solution simple au probl\`{e}me de Skorokhod.}{\em Sem. de Prob. XIII}, 90-115.

\bibitem{Barberis:14} Barberis N., 2012, A Model of Casino Gambling, {\em Management Science}, 58, 35-51.

\bibitem{CamererHo:94} Camerer, C. F., and T. Ho, 1994, Violations of the betweenness axiom and nonlinearity in probabilities, {\em Journal of Risk and Uncertainty}, 8, 167-196.

\bibitem{CerreiaVioglio:15} Cerreia-Vioglio, S., D.Dillenberger, P. Ortoleva, and G. Riella, 2017, Deliberately Stochastic, Working paper, Columbia University.

\bibitem{DayanikKaratzas:03} Dayanik S. and I. Karatzas, 2003, On the optimal stopping problem for one-dimensional diffusions. {\em Stoc. Proc. \& Appl}, 107, 2, 173-212.


\bibitem{Durrett:91} Durrett R., 1991, {\em Probability: {T}heory and {E}xamples}. Wadsworth, Pacific Grove, California.

\bibitem{EbertStrack:16} Ebert S. and P. Strack, 2015, Until the Bitter End: On Prospect Theory in a Dynamic Context. {\em American Economic Review}, 105(4), 1618-1633.

\bibitem{Hall:85} Hall W.J., 1998, On the {S}korokhod embedding theorem. {\em Technical Report 33} Stanford University, Department of Statistics.

\bibitem{HendersonHobsonZeng:17} Henderson V., D. Hobson and M. Zeng, 2017, Cautious Stochastic Choice, Optimal Stopping and Deliberate Randomization, Working paper, University of Warwick.

\bibitem{HendersonHobsonTse:17}  Henderson V., Hobson D. and A.S.L.Tse, 2017, Randomized Strategies and Prospect Theory in a Dynamic Context, {\em Journal of Economic Theory}, 168, 287-300.

\bibitem{HeHuOblojZhou:17} He X., S. Hu, J. Obloj and X.Y. Zhou, 2017, Path dependent and randomized strategies in Barberis' Casino Gambling model, {\em Operations Research}, 65, 1, 97-103.

\bibitem{HirschProfettaRoynetteYor:11} Hirsch F., C. Profetta, B. Roynette and M. Yor, 2011, Constructing self-similar martingales via two Skorokhod embeddings.
{\em Sem. de Prob. XLIII} 451-503 LNM 2006, Springer-Verlag, Berlin.

\bibitem{LoomesSugden:82} Loomes G. and R. Sugden, 1982, Regret theory: An alternative theory of rational choice under uncertainty, {\em Economic Journal}, 92, 805-824.

\bibitem{Machina:85} Machina M., 1985, Stochastic Choice Functions Generated from Deterministic Preferences over Lotteries, {\em Economic Journal} , 95, 379, 575-594.

\bibitem{Quiggin:82} Quiggin J., 1982, A Theory of Anticipated Utility, {\em Journal of Economic Behaviour and Organisation}, 3, 323-343.

\bibitem{RogersWilliams:00} Rogers L.C.G. and D. Williams, 2000, {\em Diffusions, Markov Processes and Martingales: It\^{o} Calculus} Wiley, Chichester.

\bibitem{Rogozin:66} Rogozin B.A., 1996, On the distribution of functionals related to boundary problems for processes with independent increments. {\em Th. Prob. Appl.}, 11, 580-591.

\bibitem{Skorokhod:65} Skorokhod A.V., 1965, {\em Studies in the theory of random processes}, Addison-Wesley, Reading, Mass..

\bibitem{StrackViefers:17} Strack P. and P. Viefers, 2017, Too Proud to Stop: Regret in Dynamic Decisions, {\em SSRN Working paper, id2465840}.

\bibitem{TverskyKahneman:96} Tversky, A. and D. Kahneman, 1992, Advances in Prospect Theory: Cumulative Representation of Uncertainty, {\em Journal of Risk and Uncertainty}, 5, 297-323.

\bibitem{Wakker:10} Wakker P., 2010, Prospect Theory for Risk and Ambiguity, {\it Cambridge University Press}.

\bibitem{XuZhou:13} Xu Z.Q. and X.Y. Zhou, 2013, Optimal stopping under probability distortion. {\em Ann. Appl. Prob.}, 23, 1, 251-282.

\end{thebibliography}
\end{document}